\documentclass[final,3p,times]{elsarticle}
\usepackage{graphicx}
\usepackage{algorithm}
\usepackage{times, latexsym, amsmath, verbatim, graphics, psfont, epsfig, amssymb, color}
\usepackage{algpseudocode}
\usepackage{epstopdf}
\usepackage{hyperref,csquotes}
\usepackage{natbib}  
\usepackage{subcaption}   
\usepackage{float}        
\usepackage{url}

\biboptions{sort&compress}
\usepackage{adjustbox}
\usepackage{pdflscape}
\usepackage{algpseudocode}
\usepackage{longtable}
\usepackage{multirow}
\usepackage{booktabs}
\usepackage{mathrsfs}
\usepackage[figuresright]{rotating}
\newtheorem{theorem}{Theorem}[section]
\newtheorem{lemma}[theorem]{Lemma}

\newtheorem{remark}[theorem]{Remark}

\usepackage{proof}
\begin{document}
\begin{frontmatter}
	
\title{A Novel Fractional-Order Accelerated Gradient Descent Method for Nonlinear Optimization with Application to Posture Recognition}


\author[iit]{Barsha Shaw}
\ead{shawbarsha307@gmail.com}

\author[iit]{Md Abu Talhamainuddin Ansary\corref{cor1}}
\ead{md.abutalha2009@gmail.com}

\author[ktu]{Soundararajan Ganesan}
\ead{soundarmaths001@gmail.com}

\author[ktu]{Minvydas Ragulskis 
}
\ead{minvydas.ragulskis@ktu.lt}

\cortext[cor1]{Corresponding author}

\address[iit]{Department of Mathematics, Indian Institute of Technology, Jodhpur-342303, India.}
\address[ktu]{Department of Mathematical Modelling, Kaunas University of Technology, Kaunas LT-51368, Lithuania.}
\begin{abstract} {This article proposes a Caputo fractional accelerated gradient descent (CFAGD) method for unconstrained optimization problems that is applicable to both smooth and a class of non-smooth objective functions. The proposed approach incorporates an adaptive $(\beta^k)$-parameter, which is heuristically updated throughout the iterative process to improve the search direction. Furthermore, the method employs the Caputo fractional derivative together with the adaptive $(\beta^k)$-parameter, thereby preserving the memory characteristics associated with non-integer-order derivatives.} A suitable step-size is selected via an inexact line-search technique based on the Armijo condition. The central idea is to scale the step-size by a positive parameter to improve the behavior of the iterates as they approach an optimal point, thereby generating a descent sequence. {Under strong convexity and bounded Hessian assumptions, linear convergence of the proposed method is established.} Numerical validations, including neural-network-based examples, further indicate that the CFAGD method can achieve faster and more stable performance than competing approaches.
\end{abstract}
\begin{keyword}
 Unconstrained optimization, accelerated gradient descent method, Caputo fractional derivative, neural networks.
\medskip\\
\noindent\textbf{Mathematics Subject Classification (2020):}
26A33, 65K05, 90C30, 49M05
\end{keyword}
 \end{frontmatter}
\section{Introduction}
 \hspace*{0.2cm} Optimization is a core area of applied mathematics and arises in a wide range of disciplines. In many practical situations, an optimization problem is formulated as the minimization or maximization of a real-valued objective function. Optimization techniques are used in several real-world decision-making problems, including logistics and transportation planning, healthcare resource allocation, neurodynamic approaches, financial portfolio selection, agricultural planning, energy management, manufacturing scheduling, communication networks, and machine-learning model training \cite{wang2025sodas,wei2025hessian,zhao2025inertial,wang2025interactive,chen2025scalable}.  {Robust optimization techniques are a major focus in theoretical optimization and play an important role in intelligent transportation systems, autonomous vehicles, and various engineering  applications. Liu et al. \cite{liu2021trajectory} studied vehicle trajectory prediction by considering lane-crossing behavior, final-point generation, and driving styles. Their work shows that dynamic tracking models often involve complex nonlinear and non-smooth optimization problems, thereby motivating the need for efficient gradient descent-type methods.} In these applications, the aim is to improve efficiency, accuracy, productivity, and overall performance under practical constraints \cite{portfolio,supply}.

One important category of optimization problems is unconstrained optimization. In the literature, nonlinear unconstrained optimization problems have been studied using various iterative techniques. Among these methods, gradient descent is one of the earliest and most commonly used approaches. In this method, a sequence of iterates is generated by moving along the negative gradient direction. Since this direction gives the steepest local decrease of the objective function for minimization problems, the method is also known as the steepest descent method \cite{steepest1,Nocedal}. Although steepest descent performs well for well-conditioned objective functions, its convergence can be slow for poorly conditioned problems. Therefore, the selection of a suitable step-size is important for improving the efficiency of the method. Exact line search determines an optimal step-size along a given search direction. However, this procedure is often computationally expensive. For this reason, inexact line-search techniques, such as the Armijo backtracking rule, are widely used in practice \cite{Barsha2,Barsha3}. These techniques select step-sizes that ensure a sufficient decrease in the objective value and also prevent excessively large steps. Another way to improve the performance of gradient-based methods is through acceleration. In accelerated gradient methods, a positive scaling parameter is combined with the step-size to produce a modified search step. This modification can lead to a better decrease in the objective value while preserving suitable convergence behavior under appropriate assumptions. Related studies can be found in \cite{Acce1,Acce2,Acce3,J1,J2,SAGD1,SAGD2,zhao2024adaptive}.

Most classical gradient-based methods are mainly developed for smooth objective functions because they require first-order derivative information. Therefore, difficulties arise when the objective function is non-smooth and the classical gradient does not exist at some points. To deal with such problems, generalized first-order methods such as sub-gradient methods and proximal-type methods have been developed. In particular, steepest proximal descent methods are used for non-smooth optimization problems. Proximal acceleration techniques have also been proposed to improve the convergence behavior of these methods. Such approaches are useful for composite optimization problems of the form
where $f$ is smooth and $g$ may be non-smooth \cite{Proximal-Acce1,Proximal-Acce2}. Proximal and sub-gradient-based methods have been applied in several areas, including machine learning, signal processing, image reconstruction, and statistical learning \cite{App1, App2}.  {For example, in neural-network training, non-smooth activation functions such as ReLU, defined by $\phi(t)=\max\{0,t\}$, create points where the classical derivative is not defined. At such points, sub-gradient-based interpretations are commonly used. However, the sub-gradient may not be unique at non-differentiable points, and the resulting updates can be sensitive to the chosen sub-gradient and step-size. This may lead to slow or unstable progress near activation switching regions. These difficulties motivate the study of fractional-gradient-based methods as an alternative first-order framework.}

A further approach for treating non-smooth optimization problems is to use non-integer-order derivatives, commonly known as fractional derivatives, instead of relying only on classical integer-order derivatives.  {Fractional calculus provides derivative operators that are defined through integral expressions and therefore can incorporate non-local information and memory effects, which provide valuable insights into neural network dynamics, control synthesis, and more engineering applications \cite{hongguang2026exploring, zhou2026novel}. In optimization, this memory property allows the search direction to exploit not only the current local information of the objective function but also the historical information accumulated throughout the optimization process.} In contrast, classical integer-order derivatives are purely local. As a result, classical gradient-type methods such as steepest descent and conjugate-gradient methods may suffer from slow progress or zigzag-type behavior, especially for ill-conditioned problems. These considerations have motivated the use of fractional derivatives in optimization.

Among the commonly used fractional derivative operators, the Caputo and Riemann--Liouville (R--L) derivatives are widely studied. Standard definitions of the R--L derivative $(^{RL}_{a}D_{x}^{\alpha})$ and the Caputo derivative $(^{C}_{a}D_{x}^{\alpha})$ are given in \cite{20}. Monotonicity properties of these derivatives on the interval $(0,1)$ were established in \cite{20} and later extended to generalized intervals $(n,n+1), n\in\mathbb{N}$, in \cite{Barsha1}. Further theoretical developments in fractional calculus can be found in \cite{23}. In addition, fractional Taylor-type expansions, Taylor--Riemann series, and mean value theorems for multivariate functions involving the Caputo derivative were developed in \cite{MVT,27}. Based on these tools, a Caputo-type fractional gradient descent method (CFGD) was proposed in \cite{27}.   {In the present work, the Caputo derivative is chosen because it is more compatible with the classical first-order optimality framework used in gradient-based optimization. In particular, the Caputo derivative of a constant is zero, and the modified Caputo fractional gradient employed here is connected with the classical gradient through a Taylor-type expansion. In contrast, the R--L derivative may contain lower-limit-dependent singular terms, which makes the formulation of a classical stationarity condition less direct.} Further motivation for using Caputo derivatives in steepest-descent-type methods was provided in \cite{Barsha2}.

Although fractional-gradient methods have shown useful behavior in nonlinear optimization, reducing the computational cost and the number of iterations remains an important issue. Acceleration techniques are commonly used to improve the practical performance of gradient-type methods. However, an accelerated gradient framework based on Caputo fractional derivatives is still less explored, particularly for optimization problems that include both smooth and certain non-smooth objective structures. There is also scope to examine such methods in practical learning problems; related developments in neural-network and stochastic-gradient settings can be found in \cite{Neural1,Neural2, stochastic}.
 {Although CFGD methods have been studied for nonlinear optimization, the use of fractional Taylor-based scaling within an accelerated descent framework is still limited. In particular, existing Caputo fractional gradient descent schemes mainly focus on basic descent updates, while the effect of combining a Caputo fractional gradient, an acceleration scaling parameter, and an Armijo-type line search has not been sufficiently explored. The proposed method employs a local finite-memory form of the modified Caputo gradient and therefore does not require storing all previous gradients or iterates. Consequently, the memory requirement remains independent of the iteration number rather than growing as \(\mathcal{O}(k)\) memory at the $k$-th iteration.} Motivated by the successful application of Caputo fractional derivative concepts in nonlinear optimization, the present work investigates a Caputo fractional accelerated gradient descent method as an efficient framework for obtaining improved search directions and enhancing the convergence behavior of first-order optimization schemes \cite{Barsha2,Barsha3}.  {The novelty of the proposed method lies in combining the modified Caputo fractional gradient with a fractional Taylor-based scaling parameter and an Armijo-type line search into a single descent framework.}
The main contributions are summarized as follows:
\begin{itemize}
    \item An accelerated gradient descent scheme for non-linear optimization is developed by employing a Caputo fractional gradient.
    \item  {An explicit algorithm is proposed, and its convergence properties are established under strong convexity and bounded Hessian assumptions.} (see Theorems~\ref{theorem1} and \ref{theorem2}).
    
    \item The proposed method is evaluated through numerical experiments, including neural-network examples, with ideas adapted from \cite{Nature,sou2023}.
    
\end{itemize}

\par  {
The remainder of this paper is organized as follows: Section~\ref{sec1} presents the required preliminaries, Section~\ref{sec2} develops the proposed CFAGD method, Section~\ref{sec3} establishes its convergence properties, Section~\ref{sec4} reports the numerical results, and Section~\ref{sec6} concludes the paper with possible directions for future work.}

\section{Preliminaries}\label{sec1}
Consider the unconstrained optimization problem
\[
(P): \quad \min_{x\in\mathbb{R}^n}~~ f(x),
\]
where $f:\mathbb{R}^n\to\mathbb{R}$ is a continuous function. A classical method for solving $(P)$ is the steepest descent method. In this method a sequence $\{x^k\}$ is generated with an initial approximation $x^0$ and update formula $x^{k+1}=x^k+\eta_k d^k$. Here $d^k=-\nabla f(x^k)$ is a descent direction of $f$ at $x^k$, and step-size $\eta_k$ is selected using some exact/inexact line search technique. A widely used efficient inexact line search technique is the Armijo backtracking line search. A major limitation of the steepest-descent method is its slow (linear) rate of convergence. Therefore, the steepest descent method is modified further in different directions. In the case of Newton, quasi-Newton, and conjugate-gradient methods, descent direction is computed with the help of higher order information of the objective function. On the other hand, accelerated gradient descent methods, step-size $\eta_k$ is chosen in a different way to improve the decrease in the objective function. The modified update formula for the accelerated gradient descent method used in \cite{Acce1} is  $x^{k+1}=x^k+\theta_k\eta_k d^k$ where $\theta_k>0$ is used to improve the performance of the gradient descent method. However, these modifications are restricted to smooth problems only.

 \par To handle smooth functions as well as certain non-smooth functions,  {we work with the function class \(A^n[a,b]\), which consists of all functions that are \(n\)-times differentiable on \([a,b]\) and whose \((n+1)\)-st derivative is absolutely continuous on \([a,b]\), where $n$ is a non-negative integer. Then, for \(f\in A^n[a,b]\), the Caputo fractional derivative of order \(\alpha\in(n,n+1)\) at \(x\in[a,b]\) is defined as
\begin{equation*}
{}^{C}_{c}D^{\alpha}_{x} f(x)
= \frac{1}{\Gamma(n-\alpha)}\int_{c}^{x}(x-\tau)^{\,n-\alpha-1} f^{(n)}(\tau)\,d\tau.
\end{equation*}}
\par For a multivariate continuous function, the fractional gradient \cite{27} is defined, for $c\in\mathbb{R}^{n}$, as
\begin{equation*}
{}^{C}_{c}\bar{\nabla}^{\alpha}_{x} f(x)
=
\Big(
{}^{C}_{c_1}D^{\alpha}_{x_1} f(x),\,
{}^{C}_{c_2}D^{\alpha}_{x_2} f(x),\,
\dots,\,
{}^{C}_{c_n}D^{\alpha}_{x_n} f(x)
\Big)^{\top}.
\end{equation*}

\par When one relies only on the Caputo fractional Taylor series, the above fractional-gradient expression
does not directly yield a usable optimality condition. To formulate an optimality condition, a modification
of the Taylor-type expansion is therefore required. We recall the following modified Taylor formulation of
the Caputo fractional derivative.

 {\begin{lemma}\label{thm1}(\cite[Theorem 3.2]{27})
Let $f_{\alpha,\beta'}:\mathbb{R}^n\to\mathbb{R}$, with $\alpha\in(0,1)$ and $\beta=(\beta_1,...,\beta_n)\in\mathbb{R}^n$.
Assume that $f_{\alpha,\beta}$ admits the following second--order expansion about a point $\beta_1\in\mathbb{R},c\in\mathbb{R}^n$:
\begin{equation}\label{Taylor1}
f_{\alpha,\beta_1}(x)
= f(c) + \nabla f(c)^{\top}(x-c)
+ \frac{1}{2}\,C_{\alpha,\beta_1}\,(x-c)^{\top}H(\xi)\,(x-c),
\end{equation}
where
\[
C_{\alpha,\beta_1}=\frac{\Gamma(2-\alpha)\Gamma(2)}{\Gamma(3-\alpha)}+\beta_1,
\]
and $H(\xi)$ denotes the Hessian of $f_{\alpha,\beta}$ in the classical derivative sense, evaluated at some
point $\xi$ on the line segment joining $c$ and $x$.
Moreover, the classical gradient at $c$ coincides with the Caputo gradient at $c$.
Define the Caputo fractional gradient by
\begin{equation} \label{cfdk}
  {}^{C}_{c}\nabla^{\alpha}_{x} f(x)
=
\operatorname{diag}\!\Big(1+|\beta|({}^{C}_{c}\nabla^{\alpha}_{x} I(x)\Big)^{-1}
\left(
{}^{C}_{c}\nabla^{\alpha}_{x} f_{\alpha,\beta}(x)
+\operatorname{diag}(\beta|x-c|)\,{}^{C}_{c}\nabla^{1+\alpha}_{x} f_{\alpha,\beta}(x)
\right),  
\end{equation}
where $I:\mathbb{R}^n\to\mathbb{R}^n$ is the identity mapping and $\beta\in\mathbb{R}^n$.
\end{lemma}}

\par Although a multivariate mean value theorem for the Caputo fractional derivative is available
(see Theorem~15 in \cite{MVT}) and can be derived from a Taylor-type expansion, in this article we adopt
a \emph{modified} Caputo fractional derivative obtained from the classical Taylor series expansion.
With this modification, the value of the modified Caputo fractional derivative coincides with the value of
the classical derivative.

\section{An Accelerated Gradient Descent Method using Caputo Fractional Operator}\label{sec2}
This section presents a Caputo fractional accelerated gradient descent algorithm to solve $(P)$. One can observe that ${}^{C}_{c}\nabla_{x}^{\alpha} f(x^{k})$ defined in (\ref{cfdk}) uses the parameter $\beta\in\mathbb{R}^n$. Since the parameter $\beta$ affects the Caputo fractional gradient and therefore the search direction, keeping it fixed during the whole iteration process may not accurately reflect the changing local behaviour of the objective function.  {To overcome this, we update $\beta^k$ adaptively in a coordinate-wise manner using the displacement $x_i^{k+1}-x_i^k$.} In this method, using some $\beta^0\in\mathbb{R}^n$, a sequence $\{\beta^k\}$ is generated by the update formula for component $i$: \[
\beta_i^{k+1}=\begin{cases}
    \beta_i^k+0.1,& if~x_i^{k+1}- x_i^k\geq 0.5\\
    \beta_i^k-0.1,&if~ x_i^{k+1}- x_i^k \leq -0.5\\
    \beta_i^k, & otherwise.
\end{cases}
\] 
 {\begin{remark} The adaptive update of $\beta^k$ is used as a practical heuristic to adjust the fractional-gradient parameter during the iteration. The constants $0.1$ and $0.5$ are not claimed to be theoretically optimal.
To illustrate the practical role of the adaptive $\beta$ update, consider the quadratic function
\begin{equation*}
f(x)=4x_2^2+x_1^2+x_1x_2-2x_1-3x_2,
\end{equation*}
where $x=(x_1,x_2)^T$. For this function, the exact minimizer is $x^*\approx (0.86667,0.26667)^T,$ and the corresponding minimum value is $f(x^*)\approx -1.26667.$
When the fractional-gradient method is applied with a fixed $\beta$, the obtained point is $x=(0.79829,0.251271)^T,$
with function value $f(x)\approx -1.2599.$
On the other hand, when the proposed adaptive $\beta$ update is used, the method gives $
x=(0.86362,0.263382)^T,$ with function value $f(x)\approx -1.26642.$
Thus, for this simple example, the adaptive $\beta$ update gives a point closer to the exact minimizer and a function value closer to the true minimum than the fixed-$\beta$ case. This example is intended only to motivate the practical use of the adaptive rule. It is not a theoretical proof that the constants used in the adaptive update are optimal.
\end{remark}}
Within this framework, a descent direction at $x^k$ is obtained by the following: 
\[
d^{k}=-{}^{C}_{c}\nabla_{x}^{\alpha} f(x^{k};\beta^k).
\]
Next, choose a step-size $\eta_{k}>0$ by the following inexact backtracking line search technique, 
\begin{equation*}\label{equation1}
f(x^{k}+\eta_{k} d^{k})
\leq
f(x^{k})+c_{1}\eta_{k}\,{}^{C}_{c}\nabla^{\alpha}_{x} f(x^k;\beta^k)^{\top}d^{k},
\qquad 0<c_{1}<1.
\end{equation*}
The following update formula is then employed by applying a scaled update parameterized by $\varsigma > 0$
\begin{equation*}
x^{k+1}=x^{k}+\varsigma\eta_{k} d^{k}.
\end{equation*}
The parameter $\varsigma$ is defined as described below to improve the behaviour of the Caputo fractional gradient descent algorithm.

Applying the Taylor series expansion of the Caputo fractional derivative \cite{Barsha3}, we derive
\begin{equation}\label{Taylor2}
    f(x^k +\varsigma \eta_k d^k) = f(x^k) + \varsigma \eta_k{}^{C}_{c}\nabla^{\alpha}_{x} f(x^k;\beta^k){^\top} d^k + \frac{1}{2} C_{\alpha,\beta} \varsigma^2 \eta_k^2 d^k{^\top} H(\xi;\beta^k) d^k.
\end{equation}
Thus, combining \eqref{Taylor1} and \eqref{Taylor2}, we obtain
\begin{equation}\label{equation3}
   f(x^k+ \varsigma \eta_k d^k) = f(x^k +\eta_k d^k) + \mathcal{A}_k(\varsigma),
\end{equation}
where 
\begin{equation}\label{equation4}
    \mathcal{A}_k(\varsigma) = -(1-\varsigma) \eta_k d^k{^\top} d^k -\frac{1}{2} (1-\varsigma^2)C_{\alpha,\beta} \eta_k^2 d^k{^\top} H(\xi) d^k.
\end{equation}
Throughout the article, we denote $-\eta_k d^k{^\top} d^k$ and $C_{\alpha,\beta} \eta_k^2 d^k{^\top} H(\xi) d^k$ of \eqref{equation4} by $a_k$ and $b_k$ respectively. One can observe that that $a_k\geq 0$ holds for any $f$ and $b_k\geq 0$ holds for convex functions. Therefore, $\mathcal{A}_k(\varsigma)$ admits the representation 
\begin{equation*}\label{equation5}
\mathcal{A}_k(\varsigma)
= (1-\varsigma)\, a_k - \frac{1}{2}\,(1-\varsigma^2)\, b_k .
\end{equation*}

Observe that $\mathcal{A}_k(\varsigma)$ is a quadratic function of $\varsigma$. In particular, if $b_k>0$,
then $\mathcal{A}_k$ is strictly convex and therefore admits a unique minimizer. Denote $\varsigma_m :=\operatorname*{arg\min}_{\varsigma\in\mathbb{R}} \, \mathcal{A}_k(\varsigma).$ Then clearly $\varsigma_m=a_k/b_k$ and $$\min_{\varsigma\in\mathbb{R}} \mathcal{A}_k(\varsigma)=\mathcal{A}_k(\varsigma_m)=\; -\,\frac{(a_k-b_k)^2}{2b_k}
\;\le\; 0.$$
  From \eqref{equation3}, it is obtain
    \begin{align*}
        f(x^k +\varsigma_m \eta_kd^k) = f(x^k+\eta_k d^k) -\frac{(a_k-b_k)^2}{2b_k}\leq f(x^k+\eta_k d^k).
    \end{align*}
Further, if $d^k$ is a descent direction and $a_k\neq b_k$ holds then there exists $\eta_k>0$ for which the following inequalities hold:
\begin{equation*}
f(x^{k}+\varsigma_m\eta_{k}d^{k}) < f(x^{k}+\eta_{k}d^{k}) < f(x^{k}).
\end{equation*}

Therefore, using the simple modification of step-size $\eta_k$ as $\varsigma_k\eta_k$ where $\varsigma_k= \varsigma_m=\frac{a_k}{b_k}$ we get
\begin{align}
    f(x^{k+1})& = f(x^k + \varsigma_k \eta_kd^k)\notag\\
    &\leq f(x^k) - c_1 \eta_k{}^{C}_{c}\nabla^{\alpha}_{x} f(x^k;\beta^k){^\top} d^k -\frac{(a_k-b_k)^2}{2b_k}\notag\\
    &= f(x^k) - \Big(c_1 a_k+ \frac{(a_k-b_k)^2}{2b_k}\Big)\leq f(x^k).\label{equation6}
\end{align}
We recall that if $x^{k}$ is not a minimizer of $f$, then there exists a direction $d^{k}$ along which $f$ decreases locally. 
In the proposed CFAGD scheme, the step-size $\eta_{k}$ is determined via an inexact line search. 
\begin{equation}\label{eq6}
     f(x^k +\varsigma_k \eta_k d^k) \leq f(x^k) + c_1\varsigma_k\eta_k{}^{C}_{c}\nabla^{\alpha}_{x} f(x^k;\beta^k){^\top} d^k.
\end{equation}
The following lemma establishes the existence of admissible step-sizes for the CFAGD update.

\begin{lemma}\label{lem:stepsize_lb}

Let $d^k$ be descent direction of the problem $(P)$ for some non-critical point. Then Inequality (\ref{eq6}) holds for all $\eta_k>0$ sufficiently small.
\end{lemma}
\textit{Proof:~}From \eqref{Taylor2}, we have
\begin{align*}
     f(x^k +\varsigma_k \eta_k d^k) &= f(x^k) + \varsigma_k \eta_k{}^{C}_{c}\nabla^{\alpha}_{x} f(x^k;\beta^k){^\top} d^k + \frac{1}{2} C_{\alpha,\beta} \varsigma_k^2 \eta_k^2 d^k{^\top} H(\xi;\beta^k) d^k\\
   &=  f(x^k) + c_1\varsigma_k \eta_k{}^{C}_{c}\nabla^{\alpha}_{x} f(x^k;\beta^k){^\top} d^k +\mathcal{R}(x^k;\xi;\beta^k),
\end{align*}
where
\[
\mathcal{R}(x^k;\beta^k;\xi)
=(1-c_1)\varsigma_k\eta_k\,{}^{C}_{c}\nabla^{\alpha}_{x} f(x^k;\beta^k)^{\top} d^k
+\frac{1}{2} C_{\alpha,\beta}\varsigma_k^2\eta_k^2\, d^{k\top} H(\xi;\beta^k) d^k .
\]
Then,
\[
\lim_{\varsigma_k\eta_k\rightarrow 0} \frac{\mathcal{R}(x^k;\beta^k;\xi)}{\varsigma_k\eta_k}<0
\]
for a sufficiently small value of $\varsigma_k\eta_k$ for any $\varsigma_k>0$. Hence,
\[
\mathcal{R}(x^k;\xi;\beta^k)<0
\]
for sufficiently small $\eta_k>0$. Therefore, for the upper bound of $\eta_k$, we have
\begin{align*}
\varsigma_k\eta_k\Big((1-c_1)\,{}^{C}_{c}\nabla^{\alpha}_{x} f(x^k;\beta^k)^{\top} d^k
+ \frac{1}{2} C_{\alpha,\beta}\varsigma_k\eta_k\, d^{k\top} H(\xi;\beta^k) d^k\Big)\leq 0.
\end{align*}
Since $\varsigma\eta_k>0$, it follows that
\begin{align*}
0<\varsigma_k\eta_k\leq
-\frac{2(1-c_1)\,{}^{C}_{c}\nabla^{\alpha}_{x} f(x^k;\beta^k)^{\top} d^k}
{C_{\alpha,\beta}\varsigma_k^2\eta_k^2\, d^{k\top} H(\xi;\beta^k) d^k}.
\end{align*}
Thus, inequality \eqref{eq6} holds for all $\eta_k>0$, is sufficiently small.

 {\begin{remark} Generally, it is worth noting that the use of the Armijo line search may increase the computational cost, as multiple objective function evaluations may be required at each iteration.
 In the present implementation, the fractional search direction is computed once at the current iterate, and the Armijo procedure is then used only to select a suitable step-size. Moreover, since we use a local finite-memory approximation of the Caputo derivative, the method avoids full history-dependent memory integrals.
Nevertheless, the successive objective function evaluations required by the Armijo line search may impose a significant computational burden, particularly for large-scale problems and complex neural-network models. However, fixed step-size methods may not provide reliable results for all classes of problems, because their performance is highly sensitive to the chosen step-size value.
 In future work, we plan to investigate more efficient step-size strategies, such as adaptive fixed step-size rules, approximate line-search techniques, or modified Armijo-type procedures, to further reduce the computational burden while preserving the stability of the proposed method.
 \end{remark}}
\section{Algorithm and Convergence Result}\label{sec3}

The section develops an explicit algorithm based on the theoretical results established so far. 
We also provide a convergence analysis of the proposed scheme and establish its convergence rate. 
The proposed algorithm is as follows.
\begin{algorithm}[H]
\caption{CFAGD Algorithm}
\label{algorithm1}
\begin{algorithmic}
    \State \textbf{Step 1.} \textit{Initialization:} Supply $f$, an initial point 
    $x^0 \in \mathbb{R}^n$, scalars $\sigma\in(0,1)$, $c_1$ with 
    $0<c_1<1$, and a tolerance $\epsilon>0$. Set $k=0$.
    
    \State \textbf{Step 2.} \textit{Optimality check:} Compute 
    ${}^{C}_{c}\nabla_{x}^{\alpha} f$. If 
    $\|{}^{C}_{c}\nabla_{x}^{\alpha} f(x^k;\beta^k)\|<\epsilon$, 
    then stop. Otherwise, proceed to the next step.
    
    \State \textbf{Step 3.} \textit{Descent direction:} Compute a descent 
    direction $d^k$.
    \label{ste2}
    
    \State \textbf{Step 4.} \textit{Acceleration:} Compute $a_k$ and $b_k$, 
    then $\varsigma_k$.

  \State \textbf{Step 5.} \textit{Line search:} Select $\eta_k$ satisfying 
    the inexact line search technique rule \eqref{eq6}.
    \label{line_search}
    
    \State \textbf{Step 6.} \textit{Update:} Compute 
    $x^{k+1}=x^k+\varsigma_k\eta_k d^k$. Update $k\gets k+1$ and return 
    to Step 2.
\end{algorithmic}
\end{algorithm}

 {Next, note that the term $b_k$ in \eqref{equation5} depends on the Hessian, whose direct evaluation can be computationally expensive. To avoid explicit Hessian computations, we express $b_k$ in terms of a fractional-gradient difference. Let $z^k$ be such that
\(
x^k=z^k+\eta_k d^k .
\)
Then, by the Taylor-type expansion, there exists a point $\tilde{x}^k$ on the line segment joining $z^k$ and $x^k$ such that
\begin{equation*}
\begin{aligned}
f(x^k)
&= f(z^k+\eta_k d^k) \\
&= f(z^k)
+\eta_k\left({}^{C}_{c}\nabla_{x}^{\alpha} f(z^k;\beta^k)\right)^{\top} d^k 
+\frac{1}{2}C_{\alpha,\beta}\eta_k^{2}
(d^k)^{\top}H(\tilde{x}^k;\beta^k)d^k .
\end{aligned}
\end{equation*}
Similarly, applying the corresponding gradient-difference relation along the same line segment gives
\[
{}^{C}_{c}\nabla_{x}^{\alpha} f(x^k;\beta^k)
=
{}^{C}_{c}\nabla_{x}^{\alpha} f(z^k;\beta^k)
+
C_{\alpha,\beta}\eta_k H(\tilde{x}^k;\beta^k)d^k .
\]
Therefore,
\[
{}^{C}_{c}\nabla_{x}^{\alpha} f(z^k;\beta^k)
-
{}^{C}_{c}\nabla_{x}^{\alpha} f(x^k;\beta^k)
=
-C_{\alpha,\beta}\eta_k H(\tilde{x}^k;\beta^k)d^k .
\]
Define
\[
y_k:=
{}^{C}_{c}\nabla_{x}^{\alpha} f(z^k;\beta^k)
-
{}^{C}_{c}\nabla_{x}^{\alpha} f(x^k;\beta^k).
\]
Then
\[
y_k
=
-C_{\alpha,\beta}\eta_k H(\tilde{x}^k;\beta^k)d^k .
\]
Taking the inner product with $d^k$, we obtain
\[
y_k^{\top}d^k
=
-C_{\alpha,\beta}\eta_k
(d^k)^{\top}H(\tilde{x}^k;\beta^k)d^k .
\]
Multiplying both sides by $\eta_k$, we get
\[
\eta_k y_k^{\top}d^k
=
-C_{\alpha,\beta}\eta_k^{2}
(d^k)^{\top}H(\tilde{x}^k;\beta^k)d^k .
\]
Hence,
\[
C_{\alpha,\beta}\eta_k^{2}
(d^k)^{\top}H(\tilde{x}^k;\beta^k)d^k
=
-\eta_k y_k^{\top}d^k .
\]
Since
\[
b_k=
C_{\alpha,\beta}\eta_k^{2}
(d^k)^{\top}H(\tilde{x}^k;\beta^k)d^k,
\]
we finally obtain
\[
b_k=-\eta_k y_k^{\top}d^k .
\]
Thus, $b_k$ can be computed through the fractional-gradient difference $y_k$ without explicitly evaluating the Hessian.}

Finally, the convergence and the order of convergence of the proposed  fractional gradient descent method are established in the following theorems.
\begin{theorem}\label{theorem1}
Suppose $\{x^k\}$ is a sequence generated by Algorithm \ref{algorithm1}. Further suppose $f$ is strongly convex on the level set $\mathcal{L}= \{x:f(x)\leq f(x_0)\}$ and there exists $M\geq m>0$ such that $mI\leq H(x^k;\beta^k)\leq MI$ it holds for all $x^k,\beta^k$. Then any accumulation point $x^*$ of $\{x^k\}$ satisfies ${}^{C}_{c}\nabla_{x}^{\alpha} f(x^*;\beta^*)(x^*)=0$.
\end{theorem}
\textit{Proof:~}
From \eqref{equation6} we have $f(x^{k+1})\leq f(x^k)$ for all $k$. This implies $\{x^k\}\subset \mathcal{L}$. So $\{x^k\}$ is a bounded sequence. Further, $\{f(x^k)\}$ is a monotone decreasing and bounded sequence. Then it follows that $$\underset{k\rightarrow \infty}{\lim}~~ f(x^{k+1})-f(x^k)=0.$$ 
From \eqref{equation3}, for $a_k<b_k$ we obtain
\begin{equation}\label{eq10}
f(x^k+\varsigma_m \eta_k d^k)\leq f(x^k+\eta_k d^k)-\frac{(a_k-b_k)^2}{2b_k}.
\end{equation}
Moreover, by \eqref{Taylor1} and the assumption $H(x^k;\beta^k)\leq MI$, it follows that
\begin{equation}\label{eq:taylor_step}
f(x^k+\eta_k d^k)\leq f(x^k)-\eta_k\|d^k\|^2+\frac{M\eta_k^2}{2}\|d^k\|^2.
\end{equation}
Observe that $-\eta_k+\frac{M\eta_k^2}{2}\leq -\frac{\eta_k}{2}$ holds for $0\leq \eta_k\leq \frac{1}{M}$. Using this estimate in \eqref{eq:taylor_step} yields
\begin{align*}
f(x^k+\eta_k d^k)
&\leq f(x^k)-\eta_k\|d^k\|^2+\frac{M\eta_k^2}{2}\|d^k\|^2\\
&\leq f(x^k)-\frac{\eta_k}{2}\|d^k\|^2\\
&\leq f(x^k)-c_1\eta_k\|d^k\|^2,
\end{align*}
where the last inequality holds since $c_1\leq \frac{1}{2}$. Substituting the above bound into \eqref{eq10}, we obtain
\begin{equation}\label{eq9}
f(x^k+\varsigma_m \eta_k d^k)
\leq f(x^k)-c_1\eta_k\|d^k\|^2-\frac{(a_k-b_k)^2}{2b_k}.
\end{equation}

   The backtracking procedure terminates with either $\eta_k=1$ or $\eta_k\geq \frac{\sigma}{M}$, where $\sigma\in(0,1)$. This provides a lower bound on the decrease in the function $f$.  
For $\eta_k=1$, we have
\begin{equation*}
f(x^k+\eta_k d^k)\leq f(x^k)-c_1\|d^k\|^2.
\end{equation*}
Similarly for $\eta_k\geq \frac{\sigma}{M}$,
\begin{equation*}
f(x^k+\eta_k d^k)\leq f(x^k)-\frac{c_1\sigma}{M}\|d^k\|^2.
\end{equation*}
Consequently, for $0\leq \eta_k\leq \frac{1}{M}$, 
\begin{equation}\label{equation10}
f(x^k+\eta_k d^k)\leq f(x^k)-\min\Big\{c_1,\frac{c_1\sigma}{M}\Big\}\|d^k\|^2.
\end{equation}
On the other hand,
\begin{align*}
\frac{(a_k-b_k)^2}{2b_k}
&\geq \frac{\big(\eta_k\|d^k\|^2-\eta_k^2 M\|d^k\|^2\big)^2}{2\eta_k^2 M\|d^k\|^2}
= \frac{(1-\eta_k M)^2}{2M}\|d^k\|^2.
\end{align*}
Similarly, for $0\leq \eta_k\leq \frac{1}{M}$, we have
\begin{equation}\label{equation9}
\frac{(a_k-b_k)^2}{2b_k}\geq
\min\Big\{\frac{(1-M)^2}{2M},\,\frac{(1-\sigma)^2}{2M}\Big\}\|d^k\|^2.
\end{equation}
In view of \eqref{equation10} and \eqref{equation9}, define
\[
\mathcal{M}
:= \min\left\{c_1,\frac{c_1\sigma}{M}\right\}
 + \min\left\{\frac{(1-M)^2}{2M},\,\frac{(1-\sigma)^2}{2M}\right\}.
\]
 Then, from \eqref{eq9}
        \begin{equation*}
             f(x^k + \varsigma_m \eta_k d^k) \leq f(x^k)  - \mathcal{M}\|d^k\|^2.
        \end{equation*}
This implies
        \begin{equation}
            f(x^k)- f(x^{k+1}) \geq  \mathcal{M} \|d^k\|^2.\label{equation11}
        \end{equation}
This along with $\lim_{k\rightarrow \infty} f(x^{k+1})-f(x^k)=0$ implies ${}^{C}_{c}\nabla_{x}^{\alpha} f(x^k;\beta^k) \rightarrow0$ as $k\rightarrow \infty$. Since $\{x^k\}$ is a bounded sequence, it has convergent subsequences. Suppose $\{x^k\}_{k\in K}$ be a convergent subsequence converging to $x^*$. Then clearly $\{{}^{C}_{c}\nabla_{x}^{\alpha} f(x^k;\beta^k)\}_{k\in K}$ converges to ${}^{C}_{c}\nabla_{x}^{\alpha} f(x^*;\beta^*)$. This along with ${}^{C}_{c}\nabla_{x}^{\alpha} f(x^k;\beta^k) \rightarrow0$ implies ${}^{C}_{c}\nabla_{x}^{\alpha} f(x^*;\beta^*)=0$.

\begin{theorem}\label{theorem2}
Under the assumptions of Theorem \ref{theorem1} sequence $\{f(x^k)\}$ generated by the CFAGD algorithm converges to $f^*:=f(x^*)$ at least linearly.
\end{theorem}
\textit{Proof.~} From strong convexity of $f$,
\begin{align*}
    f(x^k)- f^* &\leq {}^{C}_{c}\nabla_{x}^{\alpha} f(x^k;\beta^k)^\top (x^k-x^*) -\frac{m}{2}\|x^k-x^*\|^2\notag\\
    &\leq \|{}^{C}_{c}\nabla_{x}^{\alpha} f(x^k;\beta^k)\|\, \|x^k-x^*\| -\frac{m}{2}\|x^k-x^*\|^2.
\end{align*}
By applying Young's inequality \cite{Young}, we obtain
\begin{align}
    & \|{}^{C}_{c}\nabla_{x}^{\alpha} f(x^k;\beta^k)\|^2 \geq 2m \Big(f(x^k)- f^*\Big).\label{equation12}
\end{align}
Denote $e:= 1-2m \mathcal{M}$. Clearly $e\in(0,1)$. Now, from \eqref{equation11} and \eqref{equation12}, it follows that
\begin{align*}
    f(x^{k+1})&\leq f(x^k) - \frac{1-e}{2m}\| {}^{C}_{c}\nabla_{x}^{\alpha} f(x^k;\beta^k)\|^2\\
   &\leq f(x^k) -\Big((1-e) (f(x^k)- f^*\Big),
\end{align*}
this implies $$f(x^{k+1})- f^*\leq e\Big(f(x^k)-f^*\Big).$$ Since \(c_1,\sigma,m,M \ge 0\), it follows that \(\mathcal{M}\ge 0\). To bound each possible minimum term in \(2m\mathcal{M}\), let
\begin{align*}
\mathcal{M}' := \frac{2mc_1\sigma}{M}+\frac{(1-M)^2m}{M}~~~\text{and}~~~\mathcal{M}'' := \frac{2mc_1\sigma}{M}+\frac{(1-\sigma)^2m}{M}.
\end{align*}
Since \(0\leq m\leq M\), \(0\le c_1\le \frac{1}{2}\), and termination conditions of the backtracking procedure, it follows that \(\mathcal{M}',\mathcal{M}''<1\). Therefore, we conclude that \(e<1\). 
Therefore, the convergence of $f(x^k)$ to $f^*$ depends on the inexact line search parameters and the condition number bound $\frac{M}{m}$; hence, $\{f(x^k)\}$ converges to $f^*$ at least linearly.
 {\begin{remark}
The convergence results established in Theorems \ref{theorem1} and \ref{theorem2} rely on strong convexity and bounded Hessian assumptions.
The Hessian considered in the convergence analysis is associated with the Caputo fractional-gradient mapping, not necessarily with the classical gradient. Therefore, the results may also apply to functions that are non-smooth in the classical sense, provided the Caputo fractional gradient and the corresponding fractional Hessian exist and satisfy the stated assumptions.
\end{remark}
}

\section{Numerical Examples}\label{sec4}
\subsection{Application to Composite Optimization Problems}
The proposed algorithm is applicable to a broad class of composite optimization problems arising in many fields, including machine learning, signal processing, statistics, image processing, computer vision, control, and finance. In this section, we have compared the CFAGD method with the classical sub-gradient method and the Caputo fractional gradient descent (CFGD) method developed in \cite{27}. Python-based codes are developed for each method.  Each method is executed for a set of composite optimization problems.  {ProxGD is used as the baseline because the considered problems are of composite form \(f=g+h\), where \(g\) is smooth and \(h\) is convex but non-smooth. The proximal gradient method is specifically designed for this structure: it exploits the smooth part through a gradient step and handles the non-smooth part through the proximal operator. Therefore, it provides a more theoretically appropriate and numerically reliable reference than a general-purpose sub-gradient scheme. Since the problems considered here are strongly convex, ProxGD also yields a stable reference solution \(x^\star\) for measuring the \(\ell_2\)-distance of the competing methods. In this study, the objective contains both \(\ell_{1}\)- and \(\ell_{2}\)-norm regularization terms.} Therefore, the problem is non-smooth because of the \(\ell_{1}\)-term, while it remains strongly convex due to the quadratic loss together with the \(\ell_{2}\)-term. Strong convexity guarantees the existence of a unique minimizer and allows for a stable comparison among different algorithms. Two evaluation metrics are used: (i)~the number of outer \emph{iterations}, and (ii)~the \(\ell_{2}\)-distance from a reference stationary point computed by ProxGD. Thus, we record results for each competing method; sub-gradient updates may become numerically unstable in high-dimensional or ill-conditioned settings. In particular, as the effective rank increases, the sub-gradient method can exhibit degraded numerical behavior. Hence, one objective of our experiments is to determine the rank range over which the sub-gradient method remains reliable.  {By contrast, ProxGD is well-suited for composite models: it exploits the smooth structure of the loss while treating the non-smooth term through the proximal operator. It is also widely used in the literature as a reliable reference solver. For these reasons, we use ProxGD both as a strong baseline and as the method for generating the reference point \(x^{\star}\).}

For each run, we generate a matrix \(A\in\mathbb{R}^{m\times n}\) with independent and identically distributed standard normal entries, where \(m=100\) and \(n=50\). We choose a structured sparse vector \(x_{\text{true}}\) for preventing noise and set $
b = Ax_{\text{true}} + \sigma\varepsilon$ such that $\varepsilon\sim\mathcal{N}(0,I),$
where the noise level is \(\sigma=0.05\). All methods start from the same initialization \(x^{0}\), with each entry sampled uniformly from \((-1,1)\). We repeat this setup for \(10\) problem types with case-specific choices of \(x_{\text{true}}\), and each case is run five times. The reference point \(x^{\star}\) is obtained using ProxGD, whose step-size is determined by the Armijo backtracking process. The sub-gradient method uses the full objective \(f\) together with an inexact line-search technique.

 {CFGD computes a Caputo-type fractional gradient of \(g\) using numerical quadrature and then applies an inexact line search to the full objective \(f\). Different choices of $\alpha$ were examined in the experiments, and $\alpha=0.9$ was selected as the most effective value. Since the matrices are randomly generated, we repeat the same code cell in order to avoid biased conclusions. All experiments are conducted in Google Colab using standard scientific libraries such as NumPy and SciPy. The stopping tolerance is \(10^{-4}\). The maximum number of iterations is \(1000\) for the sub-gradient method and \(200\) for CFGD and CFAGD, because every experiment converges and does not exceed \(200 \). Since the proposed CFAGD framework is designed to handle both smooth and non-smooth optimization problems, the numerical experiments are conducted on a collection of test problems containing both classes of objective functions.} 
Table~\ref{table_itr} reports the iteration counts and the $\ell_{2}$-distance to the reference stationary point for $10$ test cases, with $5$ independent runs per case. Since the Caputo fractional derivative is not well suited to the \(\ell_{\infty}\)-term, ProxGD is used instead to handle that part of the objective.
Overall, the sub-gradient method requires the most iterations and frequently hits the maximum iteration, $1000$ which we fixed here. Across all $50$ runs, its iteration counts range from $98$ (Case~6, Run~2) to $1000$, while its $\ell_{2}$-distances range from $0.1443$ (Case~6, Run~3) to $1.7629$ (Case~9, Run~3).
CFGD substantially reduces the iteration counts, ranging from $33$ (Case~2, Run~4 and Case~7, Run~4) to $200$ (e.g., Case~2, Runs~1 and~3; Case~3, Run~3). For CFGD, we fixed maximum iteration $200$.
Its $\ell_{2}$-distances range from $0.0172$ (Case~5, Run~3) to $2.2453$ (Case~4, Run~1).
CFAGD achieves the smallest iteration counts, ranging from $17$ (Case~1, Runs~4--5; Case~2, Run~4) to $61$ (Case~4, Run~2).
It also attains the smallest $\ell_{2}$distances overall, ranging from $6.512\times 10^{-4}$ (Case~1, Run~4) to $9.549\times 10^{-2}$ (Case~6, Run~4).
For instance, in Case~2, Run~1, the sub-gradient method uses $145$ iterations with distance $1.2082$, CFGD uses $200$ iterations with distance $0.04847$, and CFAGD uses $19$ iterations with distance $0.004590$.
Moreover, in Case~5 (hinge $+$ elastic-net), CFAGD yields consistently small distances
$\{0.001801,\,0.001870,\,0.001530,\,0.001150,\,0.002800\}$ across the five runs.
Figure \ref{figure1} is describe a clear picture for the first two problem results of these three methods. For the remaining problems, graphs can be plot but here we attached for the first two problems.

This study is motivated by the practical difficulty of solving strongly convex but non-smooth composite optimization problems using classical sub-gradient schemes. Such methods may be sensitive to step-size selection and may show slow progress, especially in high-dimensional settings. To address this issue, we investigate Caputo fractional methods as first-order-type approaches that may improve the descent behaviour while keeping the implementation relatively simple.

All methods are implemented under the same experimental protocol. The numerical results show that the proposed CFAGD method gives better practical performance than CFGD and the sub-gradient method on the tested problems. These results suggest that combining fractional-order information with the proposed acceleration mechanism can improve the behaviour of gradient-type methods for the considered non-smooth objectives. However, the comparison is problem-dependent, and further theoretical analysis and broader numerical studies are required.

\begin{figure}[ht]
    \centering
    \begin{subfigure}[b]{0.30\linewidth}
        \centering
        \includegraphics[width=\linewidth]{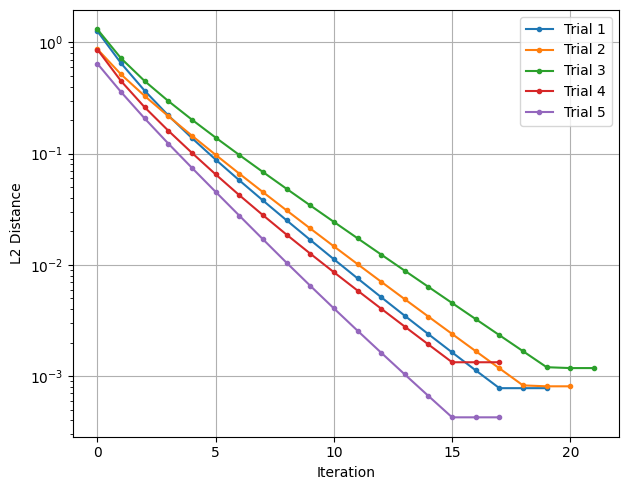}    
    \end{subfigure}
    \hfill
    \begin{subfigure}[b]{0.30\linewidth}
        \centering
        \includegraphics[width=\linewidth]{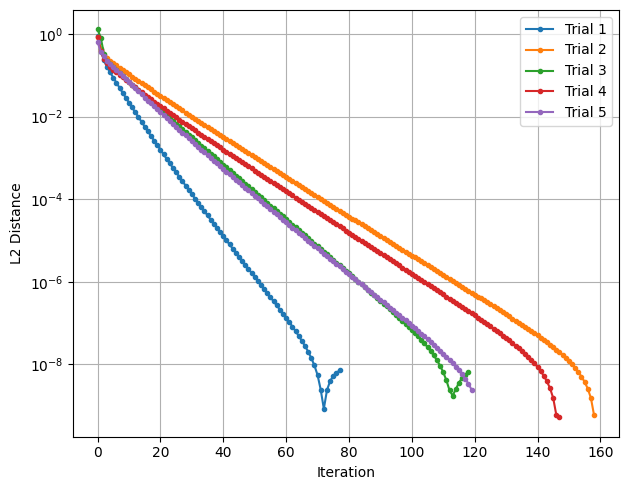} 
    \end{subfigure}
    \hfill
    \begin{subfigure}[b]{0.30\linewidth}
        \centering
        \includegraphics[width=\linewidth]{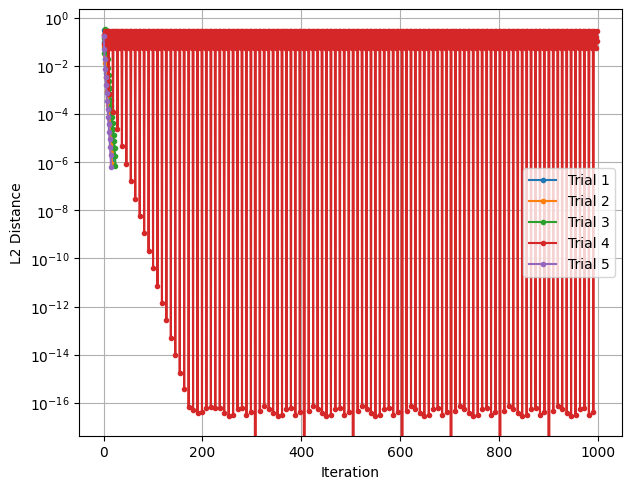}     
    \end{subfigure}

    \vspace{0.5cm}  

    \begin{subfigure}[b]{0.30\linewidth}
        \centering
        \includegraphics[width=\linewidth]{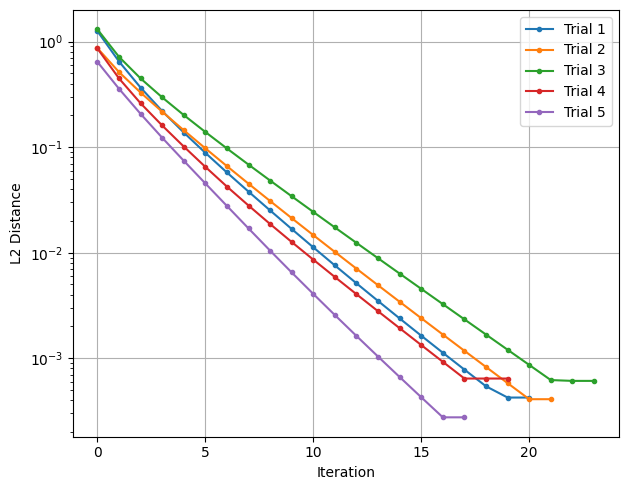}        
    \end{subfigure}
    \hfill
    \begin{subfigure}[b]{0.30\linewidth}
        \centering
        \includegraphics[width=\linewidth]{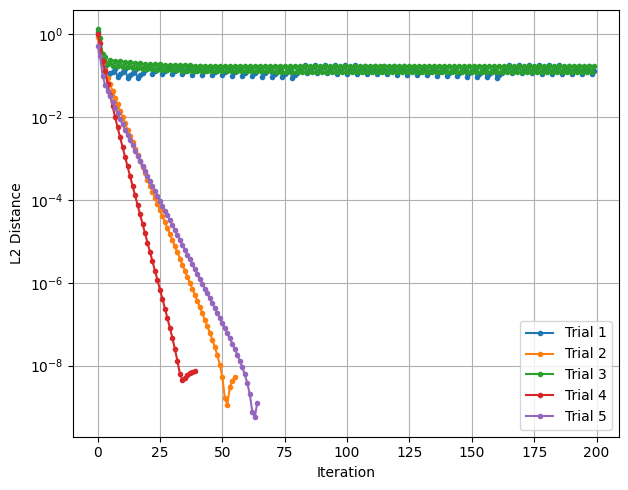}  
    \end{subfigure}
    \hfill
    \begin{subfigure}[b]{0.30\linewidth}
        \centering
        \includegraphics[width=\linewidth]{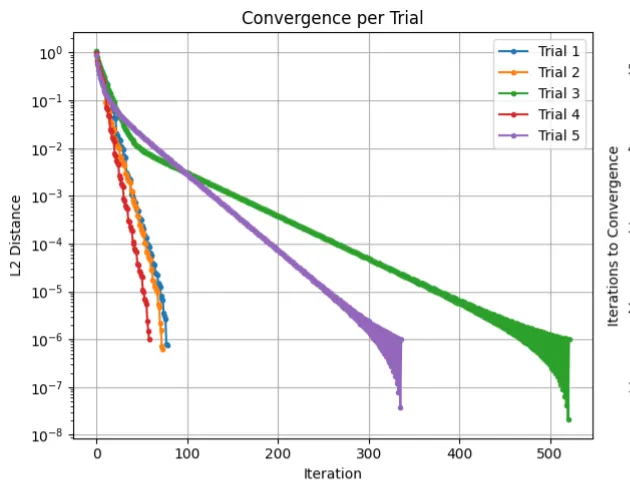}  
    \end{subfigure}

   \caption{Iterations with $\ell_{2}$ distance of CFAGD, CFGD, and sub-gradient method for elastic-net least squares and Ridge-LAD (robust $\ell_{1}$ loss) problem. }\label{figure1}
\end{figure}
\begingroup
\scriptsize
\setlength{\tabcolsep}{2pt}
\renewcommand{\arraystretch}{0.85}

\begin{longtable}{p{4.2cm}c|cc|cc|cc}
\\
\toprule
\multicolumn{2}{c|}{}
& \multicolumn{2}{c|}{Sub-gradient}
& \multicolumn{2}{c|}{CFGD}
& \multicolumn{2}{c}{CFAGD} \\
\cmidrule(lr){3-4}\cmidrule(lr){5-6}\cmidrule(lr){7-8}
Case & Run
& Iter. & $\ell_{2}$-dist.
& Iter. & $\ell_{2}$-dist.
& Iter. & $\ell_{2}$-dist. \\
\midrule
\endfirsthead

\toprule
\multicolumn{2}{c|}{}
& \multicolumn{2}{c|}{Sub-gradient}
& \multicolumn{2}{c|}{CFGD}
& \multicolumn{2}{c}{CFAGD} \\
\cmidrule(lr){3-4}\cmidrule(lr){5-6}\cmidrule(lr){7-8}
Case & Run
& Iter. & $\ell_{2}$-dist.
& Iter. & $\ell_{2}$-dist.
& Iter. & $\ell_{2}$-dist. \\
\midrule
\endhead

\midrule
\multicolumn{8}{r}{Continued on next page}
\endfoot

\endlastfoot
\multirow{5}{*}{1. Elastic-net least squares \cite{ref1}}
& 1 & 148 & 1.2076 & 70 & 0.04781 & 19 & 0.0073 \\
& 2 & 1000 & 0.9127 & 155 & 0.08372 & 19 & 0.00635 \\
& 3 & 145 & 1.1013 & 115 & 0.07545 & 21 & 0.0007232 \\
& 4 & 155 & 1.0231 & 145 & 0.08520 & 17 & 0.0006512 \\
& 5 & 1000 & 0.9985 & 120 & 0.08551 & 17 & 0.000695 \\
\midrule
\multirow{5}{*}{2. Ridge-LAD (robust $\ell_{1}$ loss)}
& 1 & 111 & 0.0070 & 200 & 0.04847 & 19 & 0.004590 \\
& 2 & 101 & 0.0071 & 51 & 0.4506 & 20 & 0.003890 \\
& 3 & 647 & 0.0070 & 200 & 00.4936 & 21 & 0.004356 \\
& 4 & 82 & 0.0070& 33 & 0.05004 & 17 & 0.004538 \\
& 5 & 415 & 0.0072 & 67 & 0.4654 & 25 & 0.004562 \\
\midrule
\multirow{5}{*}{3. Elastic-net logistic regression \cite{ref3}}
& 1 & 117 & 0.9283 & 57 & 0.04010 & 19 & 0.0066 \\
& 2 & 111 & 0.9297 & 52 & 0.04090 & 48 & 0.0031 \\
& 3 & 1000 & 1.2745 & 200 & 0.03920 & 21 & 0.0035 \\
& 4 & 105 & 0.9289 & 75 & 0.04065 & 54 & 0.0035 \\
& 5 & 1000 & 1.2801 & 62 & 0.03982 & 24 & 0.0055 \\
\midrule
\multirow{5}{*}{4. SVM (hinge) $+ \ell_{2}$ regulararization \cite{ref4}}
& 1 & 115 & 0.73212 & 79 & 2.24525 & 37 & 0.03274 \\
& 2 & 1000 & 1.7486  & 83 & 1.7938  & 61 & 0.03450 \\
& 3 & 121 & 0.7194  & 76 & 1.8876  & 34 & 0.03416 \\
& 4 & 1000 & 0.7369  & 81 & 1.9310  & 42 & 0.03278 \\
& 5 & 126 & 0.7257  & 74 & 1.9649  & 36 & 0.03387 \\
\midrule
\multirow{5}{*}{5. Hinge + elastic-net \cite{ref5}}
& 1 & 1000 & 0.9480 & 73 & 0.0184 & 27 & 0.001801 \\
& 2 & 141 & 0.94623 & 76 & 0.0201 & 22 & 0.00187 \\
& 3 & 156 & 0.94398 & 70 & 0.0172 & 28 & 0.00153 \\
& 4 & 162 & 0.94516 & 74 & 0.0189 & 22 & 0.00115 \\
& 5 & 147 & 0.94445 & 71 & 0.0196 & 23 & 0.0028 \\
\midrule
\multirow{5}{*}{6. Squared loss+ $\ell_{\infty}$ +ridge \cite{ref6}}
& 1 & 104 & 0.1465 & 46 & 0.4693 & 37 & 0.0466 \\
& 2 & 98  & 0.1521 & 44 & 0.4587 & 40 & 0.09092 \\
& 3 & 110 & 0.1443 & 48 & 0.4718 & 33 & 0.09509 \\
& 4 & 116 & 0.1497 & 45 & 0.4652 & 35 & 0.09549 \\
& 5 & 102 & 0.1460 & 47 & 0.4689 & 32 & 0.09537 \\
\midrule

\multirow{5}{*}{7. Group Lasso + ridge \cite{ref7}}
& 1 & 1000 & 0.7323 & 44 & 0.4431 & 36 & 0.001476 \\
& 2 & 1000 & 0.7248 & 62 & 0.4386 & 33 & 0.001874 \\
& 3 & 128 & 0.17371 & 45 & 0.4449 & 34 & 0.001924 \\
& 4 & 146 & 0.9294 & 33 & 0.4418 & 32 & 0.001937 \\
& 5 & 139 & 0.97339 & 44 & 0.4435 & 38 & 0.001924 \\
\midrule
\multirow{5}{*}{8. Sparse group Lasso + ridge \cite{ref8}}
& 1 & 105 & 0.81467 & 45 & 0.3723 & 37 & 0.00466 \\
& 2 & 101 & 0.91442 & 43 & 0.3667 & 43 & 0.008749 \\
& 3 & 1000  & 0.81479 & 56 & 0.3736 & 44 & 0.009245 \\
& 4 & 110 & 1.1456 & 44 & 0.3709 & 32 & 0.009372 \\
& 5 & 103 & 0.91469 & 75 & 0.3728 & 38 & 0.009241 \\
\midrule
\multirow{5}{*}{9. Fused Lasso (1D) + ridge \cite{ref9}}
& 1 & 550 & 0.7558 & 99 & 0.3545 & 26 & 0.003997 \\
& 2 & 243 & 0.97486 & 76 & 0.512 & 28 & 0.002989 \\
& 3 & 1000 & 1.7629 & 81 & 0.2368 & 29 & 0.00294 \\
& 4 & 348 & 0.7541 & 78 & 0.9837 & 29 & 0.003194 \\
& 5 & 252 & 0.7573 & 80 & 0.4651 & 27 & 0.003369 \\
\midrule
\multirow{5}{*}{10. SVM(hinge) + elastic-net}
& 1 & 190 & 0.9527 & 65 & 0.8540 & 54 & 0.002926 \\
& 2 & 186 & 0.9489 & 62 & 0.8497 & 53 & 0.002874 \\
& 3 & 194 & 0.9568 & 68 & 0.8579 & 44 & 0.002924 \\
& 4 & 189 & 0.9516 & 64 & 0.8528 & 52 & 0.002937 \\
& 5 & 192 & 0.9539 & 66 & 0.8551 & 58 & 0.003924 \\
\bottomrule
\caption{Iteration counts and $\ell_{2}$-distance from the reference stationary point by $ProxGD$ for the Sub-gradient, CFGD, and CFAGD methods.}
\label{table_itr}
\end{longtable}
\endgroup
 {\begin{remark}The present numerical comparison aims to demonstrate the effectiveness of the proposed CFAGD approach, particularly the impact of the adaptive (\(\beta^k)\)-parameter. Therefore, CFAGD is compared mainly with the existing CFGD method and the classical subgradient method under the same experimental setting. Although CFAGD involves additional computational cost per iteration due to the fractional gradient approximation and Armijo line search, it is designed to reduce the number of iterations required to reach the stopping criterion. That is, each iteration is generally more expensive than the classical gradient descent method because the fractional gradient approximation and the Armijo line search require additional computations and function evaluations. Here, the numerical results indicate that, although the proposed method has a higher cost per
iteration, it can reduce the number of iterations and improve the convergence behavior. On the other hand, ProxGD is used only as a reliable reference solver for computing the reference stationary point for the composite non-smooth problems. Hence, the results provide a focused comparison within the fractional-gradient framework, rather than a direct comparison with the proximal gradient method. Reducing the CPU time and number of function evaluations required by Caputo fractional-gradient methods will be investigated in future work.
\end{remark}}
\subsection{Application to Posture Recognition with Neural-Network Architecture}
The proposed CFAGD method is further tested on a real-life neural-network application based on a posture-recognition dataset. Instead of using standard benchmark datasets such as MNIST or CIFAR, we employ this application-driven posture dataset because it contains semantically meaningful classes that are directly relevant to real-world posture recognition\cite{zhao2023target,wang2022reslnet,khan2024robust}. In addition, the dataset exhibits realistic variability, including illumination changes, mild rotations, blur, and intra-class appearance differences, which makes it suitable for studying both generalization and robustness.

The dataset consists of four classes, namely \textit{bending}, \textit{lying}, \textit{sitting}, and \textit{standing}. All images are converted to grayscale, resized to \(64\times 64\), and normalized to the interval \([0,1]\) by dividing pixel values by \(255\). Each image is then vectorized, and the resulting image vectors are standardized before applying principal component analysis (PCA). A stratified train--test split is then used, with test size \(0.2\) and random seed \(0\).

The Caputo fractional derivative has a memory-preserving character because of its integral form, so that the resulting search direction depends not only on the current iterate but also on past optimization information. This is particularly relevant for posture-recognition data, where samples often display structured patterns, gradual variations, and correlated features across classes. In such situations, the memory effect may help smooth the optimization trajectory, reduce oscillations caused by local gradient irregularities, and improve the robustness of the updates. Therefore, the posture dataset provides a suitable testbed for assessing the practical effectiveness of the proposed optimization scheme.

The neural-network study is carried out by comparing the proposed method with CFGD (see~\cite{27}). In the first stage, the methods are compared in terms of training and testing loss, together with training and testing accuracy, during the standard learning process. In the second stage, the same performance measures are evaluated under controlled input perturbations in order to examine the empirical stability of the learned model.

A shallow neural network with one hidden layer is considered. The hidden layer has width \(H=12\) and uses the ReLU activation function, while the output layer uses the softmax function for four-class classification. Let \(x\in\mathbb{R}^d\) denote the PCA feature vector of an input sample, where \(d=30\). Then the hidden representation is given by
\begin{equation*}
a(x)=W_1x+b_1,\quad h(x)=\phi(a(x)),\quad \phi(t)=\max\{0,t\},
\end{equation*}
and the logits are defined by
\begin{equation*}
z(x)=W_2h(x)+b_2.
\end{equation*}
The predicted class label is obtained from the softmax probabilities as
\begin{equation*}
\hat y=\arg\max_{k\in\{1,\dots,K\}} p_k(x),
\end{equation*}
where \(K=4\) is the number of classes.

In the present implementation, both parts of the network are updated within each training epoch. More precisely, the hidden-layer parameters \((W_1,b_1)\) are updated by a  fractional-gradient rule based on Caputo derivatives, whereas the output-layer parameters \((W_2,b_2)\) are updated by the Adam optimizer. Thus, the hidden layer is trained through the fractional search direction and the output layer is trained through Adam. Since the full dataset is large, experiments are done for all and the numerical computation of Caputo fractional derivatives is comparatively expensive, this architecture is computationally demanding. The regularization parameter is fixed at \(\lambda=10^{-4}\), the number of epochs is \(20\), and the subset size used during training is \(512\). Here, several values of $\alpha$ were tested experimentally, and $\alpha=0.9$ was chosen as the best-performing value.

It should be emphasized that the sampled subset is used only for the parameter-update step during each epoch. In contrast, the reported training and testing losses are computed on the full training and full testing partitions after each epoch. If \(\mathcal I_{tr}\) and \(\mathcal I_{te}\) denote the index sets of the training and testing samples, then the loss on any dataset \(\mathcal I\) is evaluated by the regularized cross-entropy objective
\begin{equation*}
\mathcal L_{\mathcal I}(\theta)
=
-\frac{1}{|\mathcal I|}\sum_{i\in\mathcal I}\log p_{y_i}(x_i)
+\frac{\lambda}{2}\Big(\|W_1\|_F^2+\|W_2\|_F^2\Big),
\end{equation*}
where \(\theta=(W_1,b_1,W_2,b_2)\). Accordingly, the training loss is computed on \(\mathcal I_{tr}\), while the testing loss is computed on \(\mathcal I_{te}\).

Similarly, the training and testing accuracies are computed from the predicted labels produced by the current network after each epoch. For any index set \(\mathcal I\), the accuracy is defined by
\begin{equation*}
\mathrm{Acc}(\mathcal I)
=
\frac{1}{|\mathcal I|}\sum_{i\in\mathcal I}\mathbf{1}\{\hat y_i=y_i\},
\end{equation*}
where \(\mathbf{1}\{\hat y_i=y_i\}=1\) if the predicted label matches the true label and \(0\) otherwise. Hence, the training accuracy is obtained from the full training set, and the testing accuracy is obtained from the full test set. Therefore, the reported quantities at each epoch measure the overall performance of the current model on the complete training and testing partitions, rather than only on the sampled subset used for the update.

All experiments were carried out in Python. NumPy was used for numerical computations. scikit-learn was used for data preprocessing, including \texttt{StandardScaler}, \texttt{PCA}, and \texttt{train\_test\_split}. 
 Matplotlib was used for plotting the training curves and the final results. All experiments were performed on a desktop computer with an Intel Core i7-10700 CPU running at 2.90\,GHz, 16\,GB RAM, and Windows 11 Pro (64-bit operating system). The architecture of the network and the corresponding metric-based results are presented in the following tables and figures.
\begin{table}[htbp]
\renewcommand{\arraystretch}{1.15}
\begin{tabular}{|l|l|}
\hline
\textbf{Component} & \textbf{Setting} \\ \hline
Classes & bending, lying, sitting, standing \\ \hline
Image preprocessing & grayscale, \(64\times 64\), normalized to \([0,1]\) \\ \hline
Feature preprocessing & \texttt{StandardScaler} + PCA \\ \hline
PCA dimension (main experiment) & \(30\) \\ \hline
Train--test split & stratified, test size \(0.2\), random seed \(0\) \\ \hline
Network architecture & one hidden layer, width \(12\), ReLU + softmax \\ \hline
Adam stage & \(20\) epochs, learning rate \(10^{-4}\), batch size \(128\) \\ \hline
Regularization & \(\ell_2\), \(\lambda=10^{-4}\) \\ \hline
 Stage & output-layer fine-tuning only, \(20\) epochs, \(3\) inner steps \\ \hline
 Subset size & \(512\) \\ \hline
Fractional-order schedule & \(0.9\) \\ \hline
Hybrid parameter schedule & \(0.25 \to 0.70\) \\ \hline
Step-size rule & Armijo-type inexact line search \\ \hline
Recorded metrics & training accuracy, training loss \\ \hline

\end{tabular}
\centering
\caption{Neural-network experimental architecture for the posture dataset}
\label{tab:nn_protocol}
\end{table}

\begin{table}[ht]

\resizebox{\textwidth}{!}{
\begin{tabular}{c|cccc|cccc}
\hline
& \multicolumn{4}{c|}{\textbf{CFAGD}} & \multicolumn{4}{c}{\textbf{CFGD}} \\
\cline{2-9}
\textbf{Epoch} & \textbf{Train Loss} & \textbf{Test Loss} & \textbf{Train Acc.} & \textbf{Test Acc.} & \textbf{Train Loss} & \textbf{Test Loss} & \textbf{Train Acc.} & \textbf{Test Acc.} \\
\hline
1  & 1.3618 & 1.3626 & 0.589 & 0.592 & 1.3736 & 1.3742 & 0.564 & 0.553 \\
2  & 1.3336 & 1.3348 & 0.605 & 0.606 & 1.3596 & 1.3605 & 0.596 & 0.588 \\
3  & 1.3003 & 1.3027 & 0.621 & 0.620 & 1.3420 & 1.3436 & 0.608 & 0.607 \\
4  & 1.2614 & 1.2648 & 0.626 & 0.624 & 1.3205 & 1.3225 & 0.609 & 0.607 \\
5  & 1.2203 & 1.2252 & 0.627 & 0.630 & 1.2965 & 1.2992 & 0.609 & 0.604 \\
6  & 1.1770 & 1.1823 & 0.628 & 0.628 & 1.2692 & 1.2721 & 0.606 & 0.605 \\
7  & 1.1326 & 1.1399 & 0.634 & 0.635 & 1.2386 & 1.2425 & 0.607 & 0.610 \\
8  & 1.0880 & 1.0970 & 0.630 & 0.637 & 1.2033 & 1.2078 & 0.607 & 0.611 \\
9  & 1.0491 & 1.0605 & 0.631 & 0.634 & 1.1683 & 1.1738 & 0.605 & 0.607 \\
10 & 1.0150 & 1.0288 & 0.634 & 0.634 & 1.1341 & 1.1407 & 0.605 & 0.605 \\
11 & 0.9843 & 1.0008 & 0.638 & 0.635 & 1.1010 & 1.1091 & 0.610 & 0.606 \\
12 & 0.9563 & 0.9743 & 0.645 & 0.639 & 1.0703 & 1.0794 & 0.612 & 0.611 \\
13 & 0.9315 & 0.9524 & 0.650 & 0.647 & 1.0415 & 1.0529 & 0.615 & 0.613 \\
14 & 0.9100 & 0.9320 & 0.654 & 0.648 & 1.0159 & 1.0285 & 0.615 & 0.616 \\
15 & 0.8895 & 0.9140 & 0.658 & 0.656 & 0.9916 & 1.0061 & 0.617 & 0.619 \\
16 & 0.8687 & 0.8955 & 0.664 & 0.665 & 0.9691 & 0.9858 & 0.623 & 0.625 \\
17 & 0.8514 & 0.8797 & 0.671 & 0.669 & 0.9494 & 0.9680 & 0.626 & 0.633 \\
18 & 0.8347 & 0.8628 & 0.675 & 0.674 & 0.9316 & 0.9509 & 0.629 & 0.633 \\
19 & 0.8185 & 0.8466 & 0.680 & 0.672 & 0.9146 & 0.9346 & 0.631 & 0.636 \\
20 & 0.8047 & 0.8327 & 0.684 & 0.677 & 0.8993 & 0.9201 & 0.634 & 0.636 \\
\hline
\end{tabular}
}
\centering
\caption{Epoch-wise comparison of CFAGD and CFGD on the posture dataset.}
\label{tab3}
\end{table}
\begin{figure}[H]
\centering
\begin{subfigure}{0.45\textwidth}
    \centering
    \includegraphics[width=\linewidth]{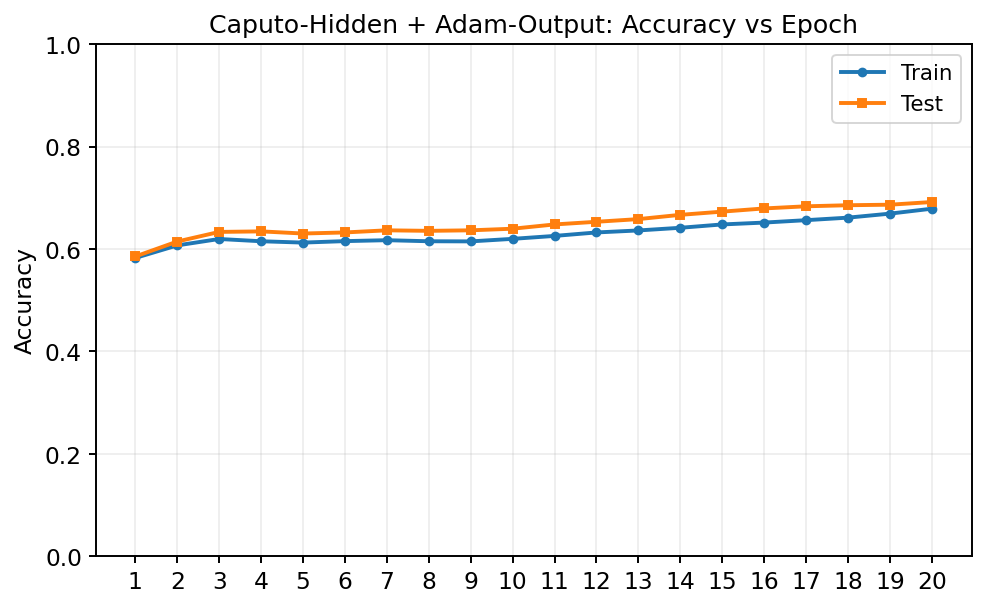}  
\end{subfigure}
\hfill
\begin{subfigure}{0.45\textwidth}
    \centering
    \includegraphics[width=\linewidth]{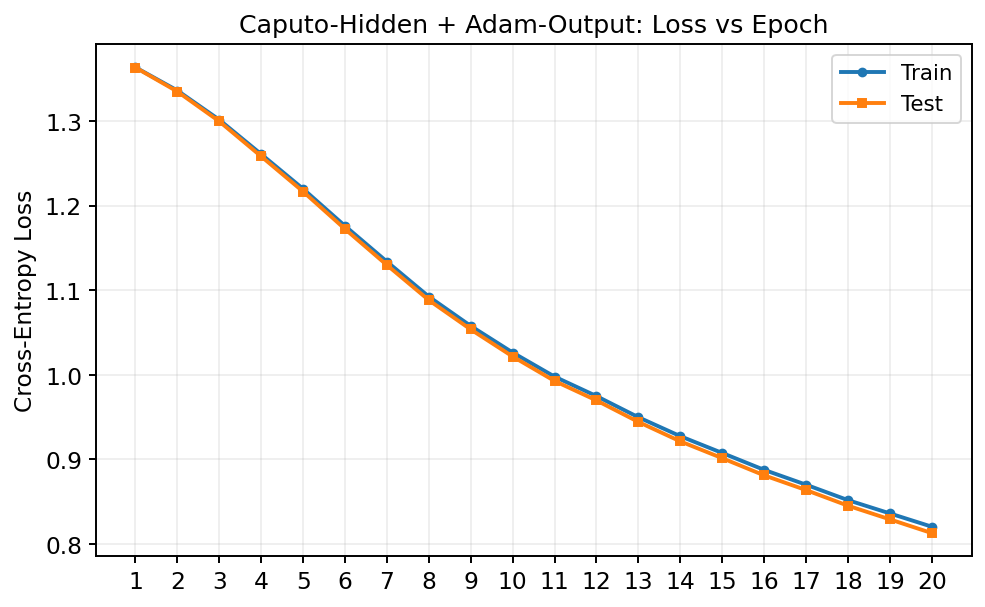}  
\end{subfigure}
\caption{Training and testing performance of CFAGD in terms of accuracy and loss.}

\begin{subfigure}{0.45\textwidth}
    \centering
    \includegraphics[width=\linewidth]{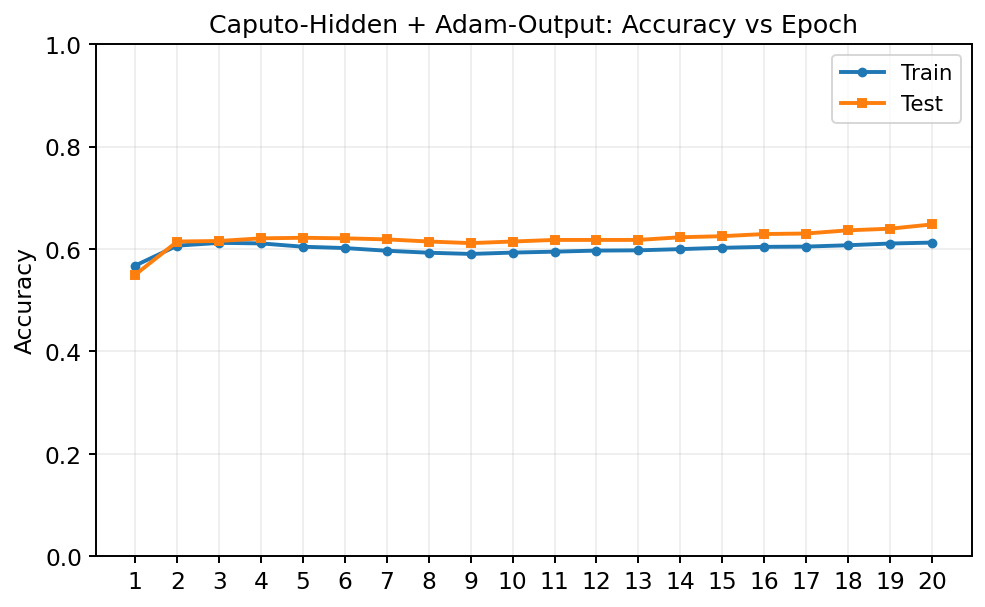}  
\end{subfigure}
\hfill
\begin{subfigure}{0.45\textwidth}
    \centering
    \includegraphics[width=\linewidth]{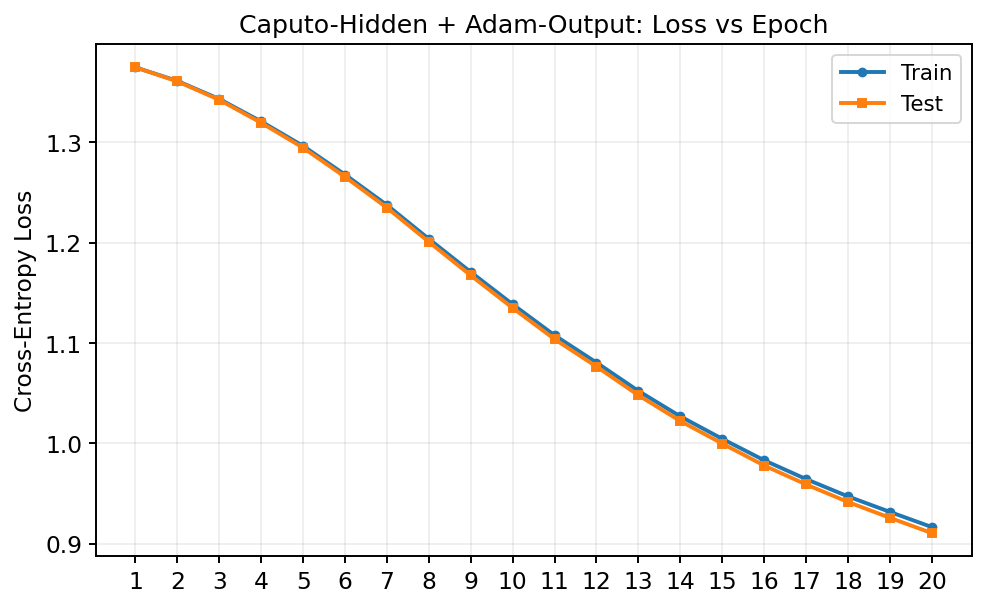}  \end{subfigure}
\caption{Training and testing performance of CFGD in terms of accuracy and loss.}
\label{fig3}
\end{figure}
Let $\hat{y}_i^{(\mathrm{orig})}$ denote the predicted label for the original test image and
$\hat{y}_i^{(\mathrm{pert})}$ denote the predicted label for its perturbed version (under a fixed perturbation type).
Prediction stability is quantified by the \emph{flip rate}:
$$
\mathrm{Flip\ rate}=\frac{1}{|\mathcal{I}_{te}|}\sum_{i\in\mathcal{I}_{te}}
\mathbf{1}\!\left\{\hat{y}_i^{(\mathrm{pert})}\neq \hat{y}_i^{(\mathrm{orig})} \right\}.
$$
\subsubsection*{Perturbation stability and flip rate:}

To assess robustness, the stability of predictions under controlled perturbations of the test images is evaluated.
For each test image, perturbed variants are generated using the following transformations and parameter ranges:
\begin{itemize}
\item \textbf{Brightness scaling:} intensities are multiplied by $\gamma_b\sim\mathcal{U}[0.85,1.15]$.
\item \textbf{Contrast scaling:} a contrast factor $\gamma_c\sim\mathcal{U}[0.85,1.15]$ is applied.
\item \textbf{Rotation:} an angle $\theta\sim\mathcal{U}[-7^\circ,7^\circ]$ is used for rotation.
\item \textbf{Gaussian blur:} a blur radius $r\sim\mathcal{U}[0.2,1.0]$ is applied.
\item \textbf{Additive Gaussian noise:} $\epsilon\sim\mathcal{N}(0,\sigma^2)$ is added with $\sigma\sim\mathcal{U}[0.01,0.03]$.
\item \textbf{Spatial shift:} a translation by $(\Delta x,\Delta y)$ is applied, where $\Delta x,\Delta y\in\{-4,-3,\dots,4\}$.
\end{itemize}

A lower flip rate indicates higher robustness, i.e., the model's prediction is less sensitive to the corresponding input perturbation. Table~\ref{tab:fliprate_alpha_methods} reports the flip rates under different perturbations for , ADAM, CFAGD and CFGD, and Figure~\ref{fig4} shows the corresponding barcode plots.
\begin{table}[H]

\renewcommand{\arraystretch}{0.78}
\setlength{\tabcolsep}{1.4pt}

\begin{adjustbox}{max totalsize={\textwidth}{0.88\textheight},center}
\begin{tabular}{l c l c c c c c}
\toprule
\multirow{2}{*}{Method} & \multirow{2}{*}{$\alpha$} & \multirow{2}{*}{Perturbation} &
\multicolumn{5}{c}{Flip Rate} \\
\cmidrule(lr){4-8}
 & & & Overall & Bending & Lying & Sitting & Standing \\
\midrule

\multirow{6}{*}{\text{ADAM+ADAM}} 
& \multirow{6}{*}{1.00} & Brightness & 0.2865 & 0.2708 & 0.1042 & 0.3875 & 0.3833 \\
&  & Contrast   & 0.0865 & 0.0667 & 0.0375 & 0.1458 & 0.0958 \\
&  & Rotate     & 0.5292 & 0.4375 & 0.3708 & 0.5708 & 0.7375 \\
&  & Blur       & 0.0010 & 0.0000 & 0.0000 & 0.0042 & 0.0000 \\
&  & Noise      & 0.1104 & 0.1083 & 0.0375 & 0.2208 & 0.0750 \\
&  & Shift      & 0.5979 & 0.4667 & 0.3417 & 0.7042 & 0.8792 \\

\midrule
\multirow{24}{*}{\text{CFAGD+ADAM}} 
& \multirow{6}{*}{0.60} & Brightness & 0.1594 & 0.2083 & 0.0375 & 0.2208 & 0.1708 \\
&  & Contrast   & 0.0490 & 0.0583 & 0.0125 & 0.1000 & 0.0250 \\
&  & Rotate     & 0.2885 & 0.4542 & 0.0417 & 0.3583 & 0.3000 \\
&  & Blur       & 0.0031 & 0.0042 & 0.0000 & 0.0042 & 0.0042 \\
&  & Noise      & 0.0573 & 0.0750 & 0.0292 & 0.0958 & 0.0292 \\
&  & Shift      & 0.5240 & 0.6917 & 0.0875 & 0.6042 & 0.7125 \\

\cmidrule(lr){2-8}
& \multirow{6}{*}{0.70} & Brightness & 0.1594 & 0.2083 & 0.0375 & 0.2208 & 0.1708 \\
&  & Contrast   & 0.0490 & 0.0583 & 0.0125 & 0.1000 & 0.0250 \\
&  & Rotate     & 0.2875 & 0.4458 & 0.0417 & 0.3625 & 0.3000 \\
&  & Blur       & 0.0042 & 0.0083 & 0.0000 & 0.0042 & 0.0042 \\
&  & Noise      & 0.0573 & 0.0750 & 0.0292 & 0.0958 & 0.0292 \\
&  & Shift      & 0.5250 & 0.6917 & 0.0875 & 0.6083 & 0.7125 \\

\cmidrule(lr){2-8}
& \multirow{6}{*}{0.80} & Brightness & 0.1604 & 0.2083 & 0.0375 & 0.2250 & 0.1708 \\
&  & Contrast   & 0.0490 & 0.0583 & 0.0125 & 0.1000 & 0.0250 \\
&  & Rotate     & 0.2875 & 0.4500 & 0.0417 & 0.3583 & 0.3000 \\
&  & Blur       & 0.0042 & 0.0083 & 0.0000 & 0.0042 & 0.0042 \\
&  & Noise      & 0.0563 & 0.0750 & 0.0292 & 0.0917 & 0.0292 \\
&  & Shift      & 0.5229 & 0.6917 & 0.0875 & 0.6042 & 0.7083 \\

\cmidrule(lr){2-8}
& \multirow{6}{*}{0.90} & Brightness & 0.1573 & 0.2042 & 0.0375 & 0.2208 & 0.1667 \\
&  & Contrast   & 0.0500 & 0.0625 & 0.0125 & 0.1000 & 0.0250 \\
&  & Rotate     & 0.2865 & 0.4417 & 0.0417 & 0.3625 & 0.3000 \\
&  & Blur       & 0.0031 & 0.0042 & 0.0000 & 0.0042 & 0.0042 \\
&  & Noise      & 0.0531 & 0.0667 & 0.0292 & 0.0875 & 0.0292 \\
&  & Shift      & 0.5240 & 0.6917 & 0.0875 & 0.6083 & 0.7083 \\

\midrule
\multirow{24}{*}{\text{CFGD+ADAM}} 
& \multirow{6}{*}{0.60} & Brightness & 0.1615 & 0.1875 & 0.0208 & 0.2208 & 0.2167 \\
&  & Contrast   & 0.0542 & 0.0875 & 0.0083 & 0.0792 & 0.0417 \\
&  & Rotate     & 0.3396 & 0.3875 & 0.0375 & 0.4500 & 0.4833 \\
&  & Blur       & 0.0021 & 0.0083 & 0.0000 & 0.0000 & 0.0000 \\
&  & Noise      & 0.0698 & 0.1208 & 0.0167 & 0.1042 & 0.0375 \\
&  & Shift      & 0.5250 & 0.5875 & 0.0875 & 0.6500 & 0.7750 \\

\cmidrule(lr){2-8}
& \multirow{6}{*}{0.70} & Brightness & 0.1604 & 0.1833 & 0.0208 & 0.2208 & 0.2167 \\
&  & Contrast   & 0.0552 & 0.0917 & 0.0083 & 0.0792 & 0.0417 \\
&  & Rotate     & 0.3406 & 0.3917 & 0.0375 & 0.4500 & 0.4833 \\
&  & Blur       & 0.0010 & 0.0042 & 0.0000 & 0.0000 & 0.0000 \\
&  & Noise      & 0.0677 & 0.1125 & 0.0167 & 0.1042 & 0.0375 \\
&  & Shift      & 0.5250 & 0.5875 & 0.0875 & 0.6500 & 0.7750 \\

\cmidrule(lr){2-8}
& \multirow{6}{*}{0.80} & Brightness & 0.1604 & 0.1833 & 0.0208 & 0.2167 & 0.2208 \\
&  & Contrast   & 0.0552 & 0.0917 & 0.0083 & 0.0792 & 0.0417 \\
&  & Rotate     & 0.3385 & 0.3875 & 0.0375 & 0.4458 & 0.4833 \\
&  & Blur       & 0.0010 & 0.0042 & 0.0000 & 0.0000 & 0.0000 \\
&  & Noise      & 0.0677 & 0.1125 & 0.0167 & 0.1042 & 0.0375 \\
&  & Shift      & 0.5250 & 0.5875 & 0.0875 & 0.6500 & 0.7750 \\

\cmidrule(lr){2-8}
& \multirow{6}{*}{0.90} & Brightness & 0.1615 & 0.1875 & 0.0208 & 0.2167 & 0.2208 \\
&  & Contrast   & 0.0552 & 0.0917 & 0.0083 & 0.0792 & 0.0417 \\
&  & Rotate     & 0.3385 & 0.3875 & 0.0375 & 0.4458 & 0.4833 \\
&  & Blur       & 0.0010 & 0.0042 & 0.0000 & 0.0000 & 0.0000 \\
&  & Noise      & 0.0677 & 0.1125 & 0.0167 & 0.1042 & 0.0375 \\
&  & Shift      & 0.5250 & 0.5833 & 0.0875 & 0.6500 & 0.7792 \\

\bottomrule

\end{tabular}

\end{adjustbox}
\centering
\caption{Flip rate under different perturbations for ADAM, CFAGD, and CFGD.}
\label{tab:fliprate_alpha_methods}
\end{table}

\begin{figure}[H]
\centering
\begin{subfigure}{0.33\textwidth}
    \centering
    \includegraphics[width=\linewidth]{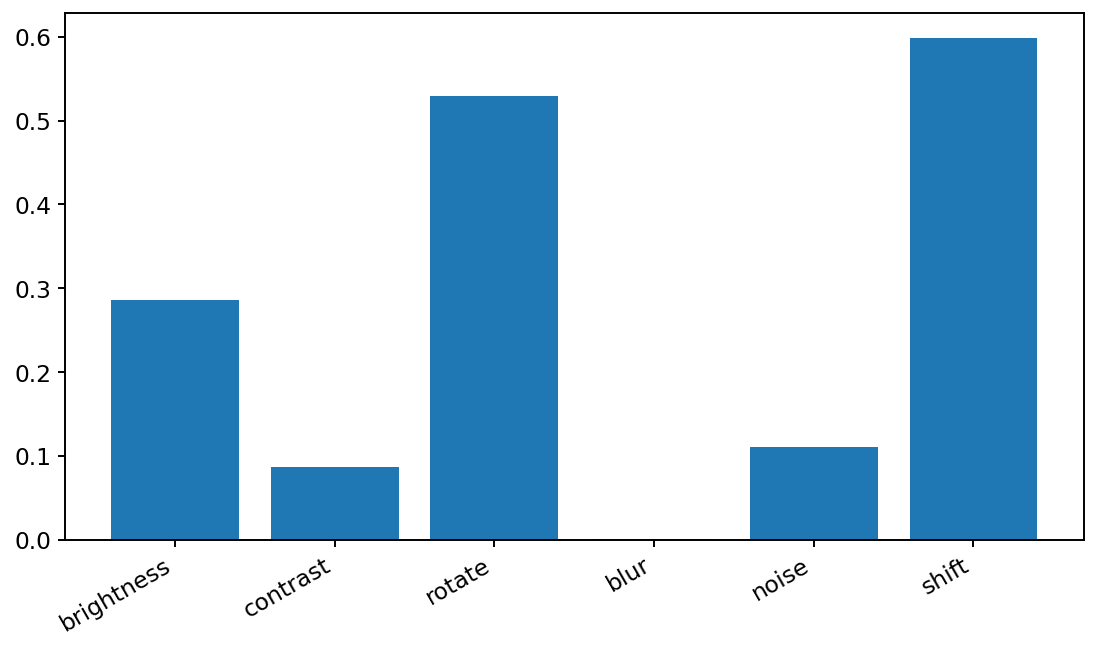}
   
\end{subfigure}
\hfill
\begin{subfigure}{0.33\textwidth}
    \centering
    \includegraphics[width=\linewidth]{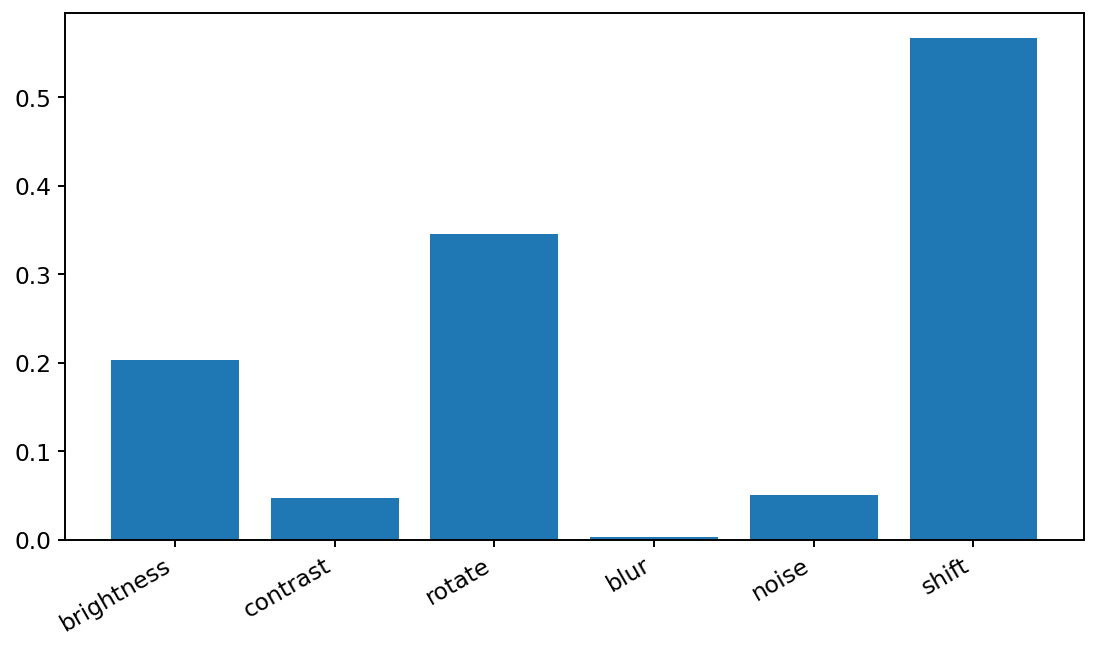}

\end{subfigure}
\hfill
\begin{subfigure}{0.33\textwidth}
    \centering
    \includegraphics[width=\linewidth]{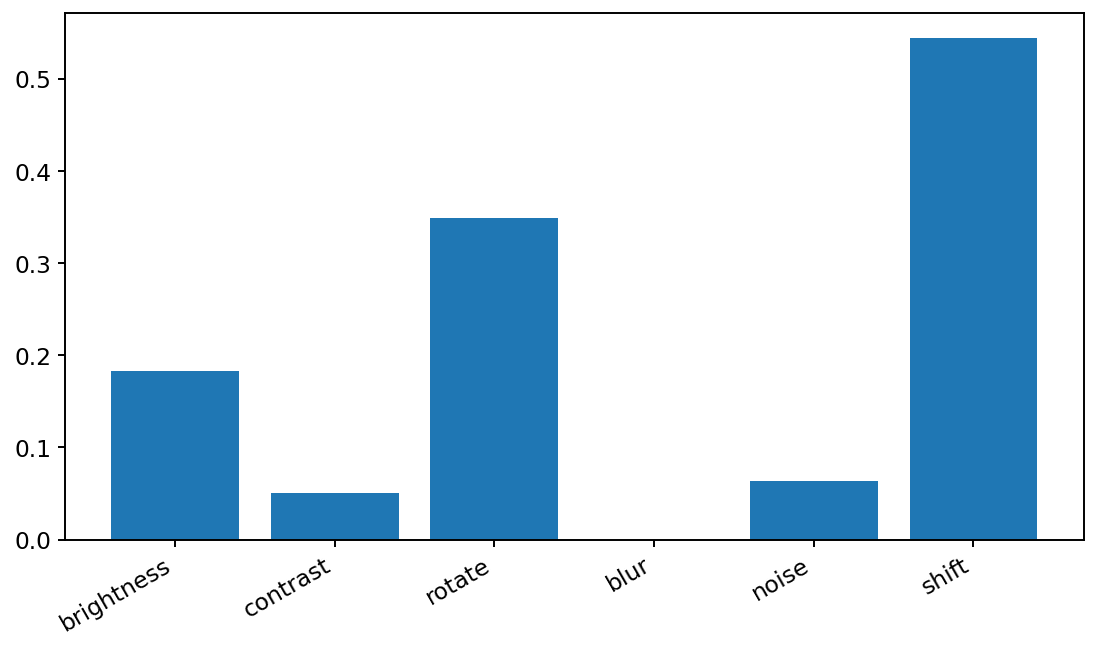}
   
\end{subfigure}

\caption{Barcode of flip rate of ADAM, CFAGD and CFGD.}
\label{fig4}
\end{figure}
 {The present experiment clearly demonstrates the positive effect of the proposed CFAGD+ADAM framework for the posture-recognition problem. Compared with the classical ADAM+ADAM model, the incorporation of Caputo fractional-gradient information improves both the learning behavior and the stability of the neural-network model. Although ADAM+ADAM is able to train the model reasonably well, its perturbation results show higher sensitivity to input distortions, particularly under brightness, rotation, noise, and spatial shift. In contrast, CFAGD+ADAM produces lower flip rates in most important perturbation cases, showing that the fractional-memory effect helps the model maintain more consistent predictions.
The advantage of CFAGD+ADAM is especially visible for $(\alpha=0.90)$. For this value, the overall flip rate decreases from 0.2865 for ADAM+ADAM to 0.1573 under brightness perturbation, from 0.5292 to 0.2865 under rotation, and from 0.1104 to 0.0531 under noise. These improvements indicate that the accelerated Caputo fractional update contributes to more stable hidden-layer representations. The memory effect of the fractional derivative helps to use past-gradient information in a more effective way, while the acceleration mechanism improves the practical behavior of the training process.
CFAGD+ADAM also performs better than CFGD+ADAM in several important cases. For $(\alpha=0.90)$, CFAGD+ADAM gives smaller overall flip rates than CFGD+ADAM under brightness, contrast, rotation, and noise perturbations. In particular, the rotation flip rate is reduced from 0.3385 for CFGD+ADAM to 0.2865 for CFAGD+ADAM, and the noise flip rate is reduced from 0.0677 to 0.0531. This confirms that the accelerated fractional strategy is more effective than the basic fractional-gradient method.
Among the tested fractional orders, $(\alpha=0.90)$ gives a good and representative performance for CFAGD+ADAM. It provides low flip rates for brightness, rotation, and noise perturbations and gives stable classification behavior. Therefore, $(\alpha=0.90)$ can be regarded as a suitable fractional order for the proposed CFAGD+ADAM model in this experiment. This value appears to balance the classical first-order gradient information with the memory effect of the Caputo fractional derivative.\\
\hspace*{0.5cm}Overall, the results support the effectiveness of CFAGD+ADAM for neural-network training on real-life posture data. The method improves robustness compared with ADAM+ADAM and also shows better stability than CFGD+ADAM in most perturbation settings. These observations indicate that the proposed accelerated Caputo fractional-gradient framework can be a useful optimization strategy for practical machine-learning problems where both accuracy and stability are important.
}

\section{Conclusion}\label{sec6}
This study proposed a Caputo fractional accelerated gradient descent method for unconstrained optimization problems, considering both smooth strongly convex problems and certain non-smooth settings. Motivated by the slow convergence and step-size sensitivity of classical sub-gradient methods, the proposed CFAGD method was developed using a Caputo fractional Taylor-series-based search direction. The existence of a descent direction and a suitable step-size was established, and the convergence behavior was studied under the stated assumptions.
The numerical experiments show that CFAGD performs effectively compared with classical sub-gradient methods and CFGD under the same setting. In particular, CFAGD often achieves better numerical performance and smaller distances to the reference stationary point computed by ProxGD. The neural-network example also demonstrates the practical effectiveness of the proposed method under perturbations.  {These results indicate that the memory-preserving scaling effect of the Caputo fractional derivative can improve the search behavior and support stable training in nonlinear optimization problems.}
 {In future work, priority will be given to reducing CPU time and improving step-size strategies. Further theoretical analysis, parameter-selection rules of statistical approaches, constrained optimization and larger practical applications will also be considered. Moreover, recent studies on spatio-temporal attention networks suggest that CFAGD could be particularly useful for optimizing sequential and temporal data models in video analysis \cite{wang2023tasta}.}
\section*{Declaration}
\begin{itemize}
\item{\bf Acknowledgement:} The authors would like to thank the reviewers
for their detailed comments and valuable suggestions, which have
significantly improved both the content and the presentation of the
results in this paper.
    \item {\bf Funding information:} Author Barsha Shaw appreciates the financial support received from the University Grants Commission (UGC) Fellowship (ID No. 211610020050), Government of India, which supported her Ph.D. work.
    \item {\bf Ethics approval and consent to participate:} Not applicable.
\item {\bf Consent for publication:} The authors declare no consent for publication.
\item {\bf Data availability statement:} The authors have used the posture project from \url{kaggle.com}.
\item {\bf Materials availability:} The simulation codes are available upon reasonable request from the corresponding author.
\item {\bf Competing interest:} The authors declare no conflict of interest.
\end{itemize}

\bibliographystyle{elsarticle-num}
\bibliography{biblo}
 \end{document}